\documentclass[12pt]{amsart}
\usepackage{enumerate,amssymb}
\usepackage[pdftex, colorlinks]{hyperref}
\hypersetup{colorlinks=true,  
    linkcolor=blue,          
    citecolor=black,        
    filecolor=black,      
    urlcolor=black           
    }

\newtheorem{thm}{Theorem}
\newtheorem{lem}[thm]{Lemma}
\newtheorem{cor}[thm]{Corollary}
\newtheorem*{thm*}{Theorem}

\newcounter{saveenum}

\theoremstyle{definition}

\newtheorem*{defn*}{Definition}
\newtheorem*{cor*}{Corollary}

\DeclareMathOperator{\lin}{lin}
\DeclareMathOperator{\vol}{vol}
\newcommand{\ic}{\alpha}
\DeclareMathOperator{\codim}{codim}
\DeclareMathOperator{\proj}{Proj}

\newcommand{\N}{\mathbb N}
\newcommand{\R}{\mathbb R}

\newcommand{\LL}{\mathcal L}
\newcommand{\PP}{\mathbb P}

\newcommand{\mult}{m} 
\newcommand{\D}{\mathcal{M}} 

\newcommand{\mc}[1]{\mathcal{#1}} 

\newcommand{\del}{\delta}

\newcommand{\ep}{\epsilon}
\newcommand{\sig}{\sigma}
\newcommand{\ka}{\kappa}
\newcommand{\lam}{\lambda} 

\newcommand{\Om}{\Omega}
\newcommand{\om}{\omega}
\newcommand{\paeom}{\PP\text{-a.e. } \om \in \Om}

\begin{document}

\title[Multiplicative ergodic theorem on Banach spaces]{A concise proof of the 
multiplicative ergodic theorem on Banach spaces}
\author{Cecilia Gonz\'alez-Tokman and Anthony Quas}
\address[Gonz\'alez-Tokman]{School of Mathematics and Physics, The University of 
Queensland,  St Lucia, QLD, 4072, Australia} 
\address[Quas]{Department of Mathematics and
Statistics, University of Victoria, Victoria, BC, CANADA, V8W 3R4}

\begin{abstract}
We give a new proof of a multiplicative ergodic theorem for quasi-compact
operators on Banach spaces with a separable dual. Our proof works by constructing
the finite-codimensional `slow' subspaces (the subspaces where the growth 
rate is dominated by some 
$\lambda_i$), in contrast with earlier infinite-dimensional multiplicative ergodic
theorems which work by constructing the finite-dimensional fast subspaces. 
As an important consequence for applications, we are able to get rid of
the injectivity requirements that appear in earlier works.
\end{abstract}

\maketitle

\section{Introduction}

The multiplicative ergodic theorem (MET) is a very powerful result 
in ergodic theory establishing the existence of \textsl{generalized eigenspaces} 
for stationary compositions of linear operators. It is of great 
interest in many areas of mathematics, including
analysis, geometry and applications. The MET was first established by 
Oseledets \cite{Oseledets} in the context of matrix cocycles.
The decomposition into generalized eigenspaces is called the 
\textsl{Oseledets splitting}.

After the original version, the MET was proved by a different
method by Raghunathan \cite{Raghunathan}.
The result was subsequently generalized to compact operators on
Hilbert spaces by Ruelle \cite{Ruelle}.
Ma\~n\'e \cite{Mane} proved a version for compact operators on Banach spaces 
under some continuity assumptions on the base dynamics and the dependence of 
the operator on the base
point. Thieullen \cite{Thieullen} extended this to quasi-compact operators.
Recently, Lian and Lu \cite{LianLu} proved a version in the context of 
linear operators on separable Banach spaces, in which the continuity 
assumption was relaxed to a measurability condition.

We prove a non-invertible Oseledets theorem (i.e. we obtain a filtration)
for a random dynamical system (the full definition is below)
acting on a Banach space with separable dual. We do not make any
assumption about injectivity of the operators, unlike most previous
Banach-space valued versions of the Multiplicative Ergodic Theorem. 
We also prove a semi-invertible Oseledets theorem (i.e. we obtain a
splitting) under the assumption that the underlying Banach space
is separable and reflexive. 

An important feature of the present approach is its constructive nature.
Indeed, it provides a robust way of approximating the Oseledets splitting, 
following what could be considered a \textsl{power method} type strategy.
This makes the work also relevant from an applications perspective.

The approach of this work is similar in spirit to that of Raghunathan,
in that we primarily work with the `slow Oseledets spaces'.
Ma\~n\'e's proof works hard to build the fast space, as
do the subsequent works based on Ma\~n\'e's template. 
These proofs rely on injectivity of the operators; some of them make use of natural extensions 
to extend the result to non-invertible operators -- this was the strategy in
\cite{Thieullen}, and it was also used by Doan in \cite{Doan} to extend
\cite{LianLu} to the non-invertible context.
In contrast, we establish the non-invertible version first and recover the 
(semi-)invertible one, including the `fast spaces', straightforwardly using duality. 
Another key simplifying feature of our method
is that we prove measurability at the end of the proof, rather than working to ensure that
all intermediate constructions are measurable.

While Raghunathan's proof uses singular value decomposition and hence relies 
on the notion of orthogonality,
we study instead collections of vectors with maximal volume growth.
Another important difference with Raghunathan's approach is that instead of 
dealing with the 
exterior algebra, we work with the Grassmannian. We claim this is
more natural since subspaces correspond  to \emph{rank one} elements of the 
exterior algebra (those that can be expressed as $v_1\wedge\ldots\wedge v_k$).
In the Euclidean setting, rank one elements naturally appear as eigenvectors of
$\Lambda^k(A^*A)$, but this does not seem to generalize to the Banach space case.
%
%
%
 
Section~\ref{sec:volume} analyses notions of volume growth for bounded linear maps 
$T$ on a Banach space $X$.
We establish an asymptotic equivalence between $k$-dimensional volume growth under
$T$ and $T^*$, as well as other measures of volume growth and Section \ref{sec:rds} 
uses these results to obtain the multiplicative ergodic theorems. 
The main results in this article are Theorem \ref{thm:Osel}
and Corollary \ref{cor:splitting}. After submitting the current article, we 
learned of an independent proof of essentially the same result via closely related
methods due to Blumenthal \cite{Blumenthal}.

\section{Volume calculations in Banach spaces}\label{sec:volume}

Let $X$ be a 
Banach space with norm $\|\cdot\|$. As usual, given a
non-empty subset $A$ of $X$ and a point $x\in X$, we define
$d(x,A)=\inf_{y\in A}d(x,y)$.  We denote by $B_X$ and $S_X$ the unit
ball and unit sphere in $X$, respectively.  The linear span of a
finite collection $C$ of vectors in $X$ will be denoted by $\lin(C)$ with the
convention that $\lin(\emptyset)=\{0\}$. The dual of $X$ will be
denoted by $X^*$. In this section, we study the relationships 
between various notions of volume and singular value for maps of Banach spaces.
Other closely related notions are due to Gelfand and Kolmogorov and are
described in Pisier's book \cite{Pisier}. For the purposes of 
later sections, it will suffice to show that two quantities agree up to 
a bounded multiplicative factor. We make no attempt to optimize the bounds. 
We use the notation $Q\asymp Q'$ if the ratio of the quantities $Q$ and $Q'$ is
bounded above and below by constants independent of the Banach space(s).

We define the $k$-\textsl{dimensional volume} of a collection, $(v_1,\ldots,v_k)$,
of vectors in a Banach space by
\begin{equation*}
  \vol_k(v_1,\ldots,v_k)=\prod_{i=1}^k d(v_i,\lin(\{v_j\colon j<i\})).
\end{equation*}

It is easy to see that $\vol_k(\alpha_1v_1,\ldots \alpha_k
v_k)=|\alpha_1|\ldots|\alpha_k|\vol_k(v_1,\ldots,v_k)$. In the case
where the normed space is Euclidean this notion corresponds with the
standard notion of $k$-dimensional volume.  Notice that
$\vol_k(v_1,\ldots,v_k)$ is not generally invariant under permutation
of the vectors.

Given a bounded linear map $T$ from $X$ to $Y$, we define
$d_kT(v_1,\ldots,v_k)$ to be $\vol_k(Tv_1,\ldots,Tv_k)$ and
$D_kT=\sup_{\|v_1\|=1,\ldots,\|v_k\|=1}d_kT(v_1,\ldots v_k)$.

\begin{lem}[Submultiplicativity]\label{lem:submult}
  Let $T\colon X\to Y$ and $S\colon Y\to Z$ be linear maps. Then
  $D_k(S\circ T)\le D_k(S)D_k(T)$.
\end{lem}

\begin{proof}
  Let $T(v_1),\ldots,T(v_k)\in X$ be linearly independent. 
  Then one checks from the definition that
  for any collection of coefficients $(\alpha_{ij})_{j<i}$, the
  following holds
  \begin{equation}\label{eq:changeByLinComb}
    \vol_k(v_1,\ldots,v_k)=\vol_k(v_1,v_2-\alpha_{21}v_1,\ldots,v_k-\sum_{j<k}\alpha_{kj}v_j). 
  \end{equation}

  Since the linear spans in the definition of volume are
  finite-dimensional spaces, the minima are attained so that $d_k
  T(v_1,\ldots,v_k)=\|T(v_1)\|\|T(v_2)-\alpha_{21}T(v_1)\|\ldots
  \|T(v_k)-\alpha_{k1}T(v_1)-\ldots-\alpha_{k,{k-1}}T(v_{k-1})\|$ for
  appropriate choices of $(\alpha_{ij})_{j<i}$. 

  Let $w_j=v_j-\sum_{i<j}\alpha_{ji}v_i$ so that
  $d_kT(v_1,\ldots,v_k)=\|T(w_1)\|\ldots \|T(w_k)\|$ and set
  $u_j=T(w_j)/\|T(w_j)\|$. Using \eqref{eq:changeByLinComb}, we have
  \begin{align*}
    d_k(S\circ T)(v_1,\ldots,v_k)&=\vol_k(ST(v_1),\ldots,ST(v_k))\\
    &=
    \vol_k(ST(w_1),\ldots,ST(w_k))\\
    &=\|T(w_1)\|\ldots\|T(w_k)\|\vol_k(S(u_1),\ldots,S(u_k)) \\
    &\le d_k T(v_1,\ldots,v_k)D_k S.
  \end{align*}
  Taking a supremum over $v_1,\ldots,v_k$ in the unit ball of $X$, one
  obtains the bound $D_k(S\circ T)\le D_k(S)D_k(T)$ as required.
\end{proof}

\begin{lem}\label{lem:unifexpDk}
Let $T\colon X\to Y$ be linear. Suppose that $V$ is a $k$-dimensional
subspace and $\|Tx\|\ge M\|x\|$ for all $x\in V$. Then $D_kT\ge M^k$.
\end{lem}

\begin{proof}
Let $v_1,\ldots,v_k$ belong to $V\cap S_X$ and satisfy 
$d(v_j,\lin(\{v_i\colon i<j\}))=1$.
Then $d_kT(v_1,\ldots,v_k)\ge M^k$.
\end{proof}

We now proceed to compare volume estimates for a linear operator $T\colon X\to Y$
and its dual $T^*\colon Y^*\to X^*$.  We introduce a third quantity to 
which we compare both
$D_k(T)$ and $D_k(T^*)$.  Given linear functionals
$\theta_1,\ldots,\theta_k \in Y^*$ and points $x_1,\ldots,x_k \in X$,
we let $U((\theta_i),(x_j))$ be the matrix with entries
$U_{ij}=\theta_i(T(x_j))$ and define
$$
E_k(T)=\sup \left\{\det U\big((\theta_i),(x_j)\big)\colon
  \|\theta_i\|=1\text{ and }\|x_j\|=1\text{ for all }i,j\right\}.
$$

\begin{lem}[Relationship between volumes for $T$ and $T^*$]\label{lem:primaldual}
  For all $k>0$, there exist positive constants $c_k$ and $C_k$ with
  the following property: For every bounded linear map $T$ from a
  Banach space $X$ to itself,
$$
c_kD_k(T)\le D_k(T^*)\le C_kD_k(T).
$$
\end{lem}

\begin{proof}
  The statement will follow from the following inequalities:
  \begin{align}
    D_k(T)&\le E_k(T)\le k!D_k(T)\label{eq:primal}\\
    D_k(T^*)&\le E_k(T)\le k!D_k(T^*).\label{eq:dual}
  \end{align}

  The second inequality of \eqref{eq:primal} is proved as follows: Let
  $x_1,\ldots,x_k$ and $\theta_1,\ldots,\theta_k$ all be of norm 1 in
  $X$ and $Y^*$ respectively. Let
  $\alpha_j=d(Tx_j,\text{lin}(Tx_1,\ldots,Tx_{j-1}))$. Let
  $c^j_1,\ldots ,c^j_{j-1}$ be chosen so that $\|Tz_j\|=\alpha_j$,
  where $z_j$ is defined by
  $z_j=x_j-(c^j_1x_1+\ldots+c^j_{j-1}x_{j-1})$. Note that
  $U'=U((\theta_i),(z_j))$ may be obtained from
  $U=U((\theta_i),(x_j))$ by column operations that leave the
  determinant unchanged. Notice also that $|U'_{ij}|=|\theta_i(Tz_j)|
  \le \alpha_j$. From the definition of a determinant, we see that
  $\det U=\det U'\le k!\alpha_1\ldots \alpha_k$. This inequality holds
  for all choices of $\theta_i$ in the unit sphere of $Y^*$. Now,
  maximizing over choices of $x_j$ in the unit sphere of $X$, we
  obtain the desired result.

  The second inequality of \eqref{eq:dual} may be obtained
  analogously.  We let
  $\beta_i=d(T^*\theta_i,\text{lin}(T^*\theta_1,\ldots,T^*\theta_{i-1}))$
  and choose linear combinations $\phi_i$ of the $\theta_i$ for which
  the minimum is obtained. The matrix $U''=U((\phi_i),(x_j))$ is
  obtained by row operations from $U$ and the
  $|U''_{i,j}|=|\phi_i(Tx_j)|=|(T^*\phi_i)(x_j)|\le \beta_i$.

  To show the first inequality of \eqref{eq:primal}, fix
  $x_1,\ldots,x_k$ of norm 1.  As before, let
  $\alpha_j=d(Tx_j,\text{lin}(Tx_1,\ldots,Tx_{j-1}))$. By the
  Hahn-Banach theorem, there exist linear functionals
  $(\theta_i)_{i=1}^k$ in $S_{Y^*}$ such that $\theta_i(Tx_i)=\alpha_i$
  and $\theta_i(x_k)=0$ for all $k<i$. Now
$$
\det U((\theta_i),(x_j))=\prod \alpha_i.
$$
Maximizing over the choice of $(x_j)$, we obtain $E_k(T)\ge D_k(T)$ as
required.

Finally, for the first inequality of \eqref{eq:dual}, we argue as
follows.  Let $\epsilon>0$ be arbitrary and let
$\theta_1,\ldots,\theta_k$ belong to the unit sphere of $Y^*$. We may
assume that $T^*\theta_1,\ldots,T^*\theta_k$ are linearly independent --
otherwise the inequality is trivial. Let $\phi_i=T^*\theta_i-\sum_{k<i}a_{ik}T^*\theta_k$
be such that $\|\phi_i\|=d(T^*\theta_i,\lin(\{T^*\theta_k\colon k<i\}))$. 
We shall pick $x_1,\ldots,x_k$ inductively in such a way that 
$|\det((\phi_i(x_j))_{i,j\le l})|$
is at least $\prod_{i=1}^l(\|\phi_i\|-\epsilon)$ for each $1\le l\le k$. 
Suppose $x_1,\ldots,x_{l-1}$ have been chosen.
Then since $\det((\phi_i(x_j))_{i,j<l})$ is non-zero, the rows span $\R^{l-1}$. 
Hence there exist $(b_i)_{i<l}$ such that $\psi_l:=\phi_l+\sum_{i<l}b_i\phi_i$
satisfies $\psi_l(x_j)=0$ for all $j<l$. By assumption, $\|\psi_l\|\ge \|\phi_l\|$.
Pick $x_l\in S_X$ such that $\psi_l(x_l)>\|\psi_l\|-\epsilon$. Then the matrix with 
a row for $\psi_l$ and a column for $x_l$ adjoined has determinant of absolute value
at
least $\prod_{i=1}^l(\|\phi_i\|-\epsilon)$. The matrix with $\phi_l$
replacing $\psi_l$ has the same determinant, completing the induction. Maximizing over 
the choice of $(\theta_j)_{j\le k}$, letting $\epsilon$ shrink to 0, and observing that
$\det((\phi_i(x_j))_{i,j\le k})=\det((T^*\theta_i(x_j))_{i,j\le k})$ completes the proof.

\end{proof}

A fourth quantity that will play a crucial role in what follows is
$F_k(T)$, defined as
$$
F_k(T)=\sup_{\dim(V)=k}\inf_{v\in V\cap S_X}\|Tv\|.
$$

We make use of the following lemma due to Gohberg and Krein whose
proof may be found in Kato's book \cite{Kato} (Chapter 4, Lemma 2.3).
\begin{lem}[Gohberg and Krein]\label{lem:GK}
Let $V_1$ be a proper finite-dimensional subspace of a subspace $V_2$ of a 
Banach space, $X$. Then there exists $v\in V_2\setminus\{0\}$ such that 
$d(v,V_1)=\|v\|$. 
\end{lem}

\begin{lem}[Relation between determinants and $F_k$]\label{lem:Frec}
  Let $T$ be a bounded linear map from a Banach space $X$ to
  a Banach space $Y$. Then
$$
E_{k-1}(T)F_k(T)\le E_k(T) \le k2^{k-1}E_{k-1}(T)F_k(T).
$$
\end{lem}

\begin{proof}
  We first show $E_k(T)\le k2^{k-1}E_{k-1}(T)F_k(T)$. We may assume
  $E_k(T)>0$ as otherwise the inequality is trivial.  Let
  $\theta_1,\ldots,\theta_k$ be elements of the unit sphere of $X^*$
  and $x_1,\ldots,x_k$ be elements of the unit sphere of $X$.  Let $U$
  be the matrix with entries $\theta_i(Tx_j)$. Assume that $\det U\ne
  0$.  Since $x_1,\ldots,x_k$ span a $k$-dimensional space, there
  exists a $v=a_1x_1+\ldots+a_kx_k$ of norm 1 such that $\|Tv\|\le
  F_k(T)$. By the triangle inequality, one of the $|a|$'s, say
  $|a_{j_0}|$, must be at least $\frac 1k$. Let $\tilde x_j=x_j$ for
  $j\ne j_0$ and $\tilde x_{j_0}=v$ and set $\tilde U$ to be the
  matrix with entries $\theta_i(T\tilde x_j)$. By properties of
  determinants, we see $|\det\tilde U| =|a_{j_0}|\,|\det U|\ge \tfrac 1k|\det
  U|$. Next, there exists $i_0$ for which $|\theta_{i_0}(Tv)|$ is
  maximal, this maximum not being 0 since $|\det\tilde U|$ is
  positive.  Let
  $\bar\theta_i=\theta_i-(\theta_i(Tv)/\theta_{i_0}(Tv))\theta_{i_0}$
  for $i\ne i_0$ and $\bar\theta_{i_0}=\theta_{i_0}$, so that
  $\|\bar\theta_i\|\le 2$ and $\bar\theta_i(Tv)=0$ for $i\ne i_0$.

  Now let $\bar U_{ij}=\bar\theta_i(T\tilde x_j)$, so that $|\det
  U|\le k |\det\tilde U| =k|\det \bar U|$. Finally, the $j_0$th column
  of $\bar U$ has a single non-zero entry that is at most $\|Tv\|\le
  F_k(T)$ in absolute value. The absolute value of the cofactor is
  $\left|\det \big(\bar \theta_i(T(\tilde x_j))\big)_{i\ne i_0,\ j\ne
      j_0}\right| \le 2^{k-1}E_{k-1}(T)$.  Taking a supremum over
  choices of $(\theta_i)$ and $(x_j)$, we have shown $E_k(T)\le
  k2^{k-1}E_{k-1}(T)F_k(T)$.

  For the other inequality, we may suppose that $T$ has kernel of
  codimension at least $k$, otherwise $F_k(T)=0$ and there is nothing
  to prove. Let $\theta_1,\ldots,\theta_{k-1}$ and
  $x_1,\ldots,x_{k-1}$ be arbitrary.  Let $\Delta$ be the determinant
  of the matrix with entries $\theta_i(Tx_j)$.  Let $V$ be a
  $k$-dimensional subspace such that $V\cap \ker T=\{0\}$. Let
  $W=\lin(Tx_1,\ldots,Tx_{k-1})$. Using Lemma \ref{lem:GK}, let
  $z$ be a point in the unit
  sphere of $T(V)$ such that $d(z,W)=1$.  Let
  $v\in V\cap S_X$ be such that $T(v)$ is a multiple of $z$.  Let
  $\theta_k$ be a linear functional of norm 1 such that
  $\theta_k\vert_W=0$ and $\theta_k(z)=1$ and let $x_k=v$. Now forming
  the $k\times k$ matrix $\big(\theta_i(x_j)\big)_{1\le i,j\le k}$, we
  see the absolute value of the determinant is $\Delta\cdot
  \theta_k(Tv)= \Delta\cdot\|Tv\| \ge\Delta\cdot \inf_{x\in V\cap
    S_X}\|Tx\|$.  Taking suprema over choices of $x$'s, $\theta$'s and
  $k$-dimensional $V$'s, we see that $E_k(T)\ge F_k(T)E_{k-1}(T)$ as
  required.

\end{proof}

\begin{cor}\label{cor:equiv}[of Lemmas \ref{lem:primaldual} and \ref{lem:Frec}]
  For each $k>0$, the quantities $D_k(T)$, $D_k(T^*)$, $E_k(T)$ and
  $\prod_{i\le k}F_i(T)$ agree up to multiplicative factors that may
  be bounded by constants independent of the bounded linear map $T$ 
  and the Banach spaces $X$ and $Y$. Further, $F_i(T)$ and $F_i(T^*)$ agree 
  up to a uniformly bounded multiplicative factor. 
\end{cor}

We comment that besides these approximate Banach space versions of 
singular values, additional related quantities are given by Gelfand numbers
and Kolmogorov numbers (see the book of Pisier \cite{Pisier} for more information). 
It can be checked that these quantities also agree with the sequence of $F_i$'s
up to bounded multiplicative factors (dependent on $i$, but independent of 
$X$ and $T$). 

By definition, for each natural number $k$, one can find sequences
$(\theta_i)_{i\le k}$ and $(x_j)_{j\le k}$ such that
$\det(U((\theta_i),(x_j)))\asymp E_k(T)$.  We now show that we can
find infinite sequences $(\theta_i)$ and $(x_j)$ so that, for each
$k$, $\det(U((\theta_i)_{i\le k},(x_j)_{i\le k}))\asymp E_k(T)$.

\begin{lem}[Existence of consistent sequences]\label{lem:genseqs}
  Let $X$ and $Y$ be infinite-dimensional Banach spaces. 
  For any linear map $T\colon X\to Y$, there exist $(\theta_i)_{i\ge 1}$
  in $S_{Y^*}$ and $(x_j)_{j\ge 1}$ in $S_X$ such that for all $k$,
  \begin{align*}
    &\det\left((\theta_i(Tx_j))_{1\le i,j\le k}\right)\ge
    \tfrac{1}{2^k}\prod_{i\le k}F_i(T)
    \text{; and}\\
    &\|Tx\|\ge 4^{-k}F_k(T)\|x\|\text{ for all $x\in
      \lin(x_1,\ldots,x_k)$.}
  \end{align*}
\end{lem}

\begin{proof}
  The proof is by induction: suppose $(\theta_i)_{i<k}$ and
  $(x_j)_{j<k}$ have been chosen and satisfy the desired inequalities
  at stage $k-1$.  Then pick an arbitrary $k$-dimensional space $V$
  such that $\|Tv\|\ge \frac 12F_k(T)\|v\|$ for all $v\in V$. Using
  Lemma \ref{lem:GK}, let $x_k\in V\cap S_X$ be such that
  $d(Tx_k,\lin(Tx_1,\ldots,Tx_{k-1}))= \|Tx_k\|$.  Finally choose
  $\theta_k$ of norm 1 such that $\theta_k(Tx_i)=0$ for $i<k$ and
  $\theta_k(Tx_k)=\|Tx_k\|$. The determinant inequality at stage $k$
  follows.

  Let $x=a_1x_1+\ldots+a_kx_k$ be of norm 1. Then
  \begin{equation}\label{eq:rough1}
    \|Tx\|\ge |a_k|d(Tx_k,\lin(Tx_1,\ldots,
    Tx_{k-1}))= |a_k|\|Tx_k\|\ge |a_k|F_k(T)/2.
  \end{equation}
  Also,
  \begin{equation*}
    \|Tx\|\ge \left\|T(\textstyle\sum_{j<k}a_jx_j)\right\|-|a_k|\|Tx_k\|.
  \end{equation*}
  Averaging the inequalities, we get
  \begin{equation}\label{eq:rough2}
    \|Tx\|\ge \tfrac12 \left\|T(\textstyle\sum_{j<k}a_jx_j)\right\|.
  \end{equation}
  If $|a_k|>\frac 12$, the first inequality yields $\|Tx\|\ge \frac
  14F_k(T)$.  If $|a_k|\le \frac 12$, then $\|\sum_{j<k}a_jx_j\|\ge
  \frac 12$ and the second inequality combined with the inductive
  hypothesis gives $\|Tx\|\ge \frac 14 4^{-(k-1)}F_{k-1}(T)$.

\end{proof}

\begin{lem}[Lower bound on volume growth in a subspace of finite codimension]\label{lem:ExponentsOfReducedCocycle}
  For any natural numbers $k>m$, there exists $C_k$ such that if $X,Y$
  are Banach spaces, $T\colon X\to Y$ is a linear map and $V$ is a closed
  subspace of $X$ of codimension $m$, then $D_k(T)\le
  C_kD_m(T)D_{k-m}(T|_V)$.
\end{lem}

\begin{proof}
  Let $\ep>0$.  Let $P$ be a projection from $X$ to $V$ of norm at
  most $\sqrt m+\ep$ (such a projection exists
  by Corollary III.B.11 in the book of Wojtaszczyk \cite{Wojtaszczyk}).  
  Then, $\|I -P\|\leq \sqrt m +\ep+1$.
  Let $x_1,\ldots,x_k$ be a sequence of vectors in $X$ of norm 1.  The
  proof of Lemma \ref{lem:primaldual} shows that there exist
  $\psi_1,\ldots,\psi_k$ in $S_{Y^*}$ such that $\det(\psi_i(x_j)) \ge
  d_kT(x_1,\ldots,x_k)$. Write $P_1$ for $P$ and $P_0$ for $I-P$, 
  which has $m$-dimensional range.  There exists a choice
  $\epsilon_1,\ldots\epsilon_k\in\{0,1\}^k$ such that
  $|\det(\psi_i(P_{\ep_j}x_j))|>2^{-k}d_kT(x_1,\ldots,x_k)$, by
  multilinearity of the determinant.  At most $m$ of the $\epsilon_j$
  can be 0, as otherwise more than $m$ vectors lie in a common
  $m$-dimensional space, so that at least $k-m$ of them lie in $V$.
  Hence, there exist vectors $z_1,\ldots,z_m$ in $S_X$ and
  $z_{m+1},\ldots,z_k$ in $S_X\cap V$ such that
$$|\det(\psi_i(z_j))|\ge (2(\sqrt m +\ep+1))^{-k}d_kT(x_1,\ldots,x_k).
$$
Using the proof of Lemma \ref{lem:primaldual} again, we deduce that
\begin{align*}
  d_mT(z_1,\ldots,z_m) & d_{k-m}T(z_{m+1},\ldots,z_k)\ge d_kT(z_1,\ldots,z_k)\\
  &\ge (2(\sqrt m +\ep+1))^{-k}/(k!) d_kT(x_1,\ldots,x_k).
\end{align*}
This completes the proof.
\end{proof}

\section{Random dynamical systems}\label{sec:rds}

A closed subspace $Y$ of $X$ is called \textsl{complemented} if there
exists a closed subspace $Z$ such that $X$ is the direct sum of $Y$
and $Z$, written $X=Y\oplus Z$. That is, for every $x\in X$, there
exist $y\in Y$ and $z\in Z$ such that $x=y+z$, and this decomposition
is unique.  The Grassmannian $\mathcal G(X)$ is the set of closed
complemented subspaces of $X$. We equip $\mathcal G(X)$ with the metric
$d(Y,Y')=d_H(Y\cap S_X,Y'\cap S_X)$ where $d_H$ denotes the Hausdorff
distance.  We denote by $\mathcal G^k(X)$ the
collection of closed $k$-codimensional subspaces of $X$ (these are
automatically complemented), by $\mathcal G_k(X)$ the $k$-dimensional 
subspaces of $X$. 
If $U$ and $V$ are closed subspaces of $X$ such that $U\oplus V=X$, then
$\proj_{U\parallel V}$ is the projection onto $U$ parallel to $V$
(that is $\proj_{U\parallel V}(x)\in U$ and $x-\proj_{U\parallel V}(x)\in V$). 
We record some facts about Grassmannians in the following lemma. 

\begin{lem}\label{lem:Grass}
Let $X$ be a  Banach space with separable dual.
Let $k\in\mathbb N$. 
The following facts hold:
\begin{enumerate}
\item\label{it:compsep}
$\mathcal G^k(X)$ is complete and separable.
\item
If $V\in \mathcal G(X)$, $W\in \mathcal G(X)$
and $V\oplus W=X$, then 
$\frac 1\delta\le \|\proj _{V\parallel W}\|\le \frac 2\delta$, where
$\delta=\inf_{x\in V\cap S_X,y\in W\cap S_X}\|x-y\|$. 
\label{it:projbound}
\item 
There exists $K>0$ (independent of $X$) such that if $V\in G^k(X)$, 
there exists a subspace $W\in \mathcal G_k(X)$
such that $\|\proj_{W\parallel V}\|\le K$ 
and $\|\proj_{V\parallel W}\|\le K$. \label{it:projbound2}
\item (Symmetry of closeness) There exists $K>0$ such that if 
$V,V'\in \mathcal G^k(X)$, then
$$
\sup_{v'\in V'\cap S_X}\inf_{v\in V\cap S_X}\|v-v'\|
\le K \sup_{v\in V\cap S_X}\inf_{v'\in V'\cap S_X}\|v-v'\|.
$$
\label{it:symmclose}
\end{enumerate}
\end{lem}

\begin{proof}
The map $\perp\colon\mathcal G^k(X)\to
\mathcal G_k(X^*)$ defined by $V^\perp=\{\theta\in X^*\colon\theta|_V=0\}$
is a bi-Lipschitz bijection \cite{Kato}. Separability of $\mc{G}^k(X)$
and symmetry of closeness
are proved in \cite{GTQuas}. The completeness is stated but not
proved in Kato's book. We sketch a proof using results from the appendix of \cite{GTQuas}. 
Let $V$ be a $k$-dimensional subspace 
of a Banach space $Z$ and let $v_1\ldots,v_k$ be an Auerbach basis. 
By the Hahn-Banach theorem, there exist $\theta_1,\ldots,\theta_k\in Z^*$
of norm 1 such that $\theta_i(v_j)=\delta_{ij}$. Now if $\tilde v_1,\ldots,
\tilde v_k$ in $Z$ satisfy $\|\tilde v_i-v_i\|<\epsilon/k$ for each $k$, then
one has $\|\sum a_i\tilde v_i\|\ge (1-\ep)\max|a_i|$ (to see this, apply 
$\theta_{i_0}$ where $|a_{i_0}|=\max |a_i|$). From this, we see
that $\tilde v_1,\ldots,\tilde v_k$ is an $\epsilon$-nice basis (as defined in
\cite{GTQuas}). Now let $(V_n)$ be a Cauchy sequence in $\mc G_k(Z)$. 
By refining the sequence, one may assume $d(V_n,V_{n+1})<(3k)^{-n}$. 
Choosing an Auerbach basis $v^1_1,\ldots, v^1_k$ for $V_1$, one may
then obtain elements $v^n_1,\ldots,v^n_k$ 
of $V_n$ satisfying $\|v^{n+1}_i-v^n_i\|<(2k+1)^{-n}$. 
This is a convergent sequence of $\frac 12$-nice bases. Letting $v^*_i$
be the limit of $v^n_i$, Corollary B6 of \cite{GTQuas} shows that
$d(V_n,V_*)\to 0$, where $V_*$ is the subspace spanned by the $v^*_i$. 
This establishes completeness of $\mc{G}_k(X^*)$ and hence 
completeness of $\mc G^k(X)$. 
To see \eqref{it:projbound}, if $v_n\in V\cap S_X$
and $w_n\in W\cap S_X$, satisfy $\|v_n-w_n\|\to \delta$ then 
$\|\proj_{V\parallel W}(v_n-w_n)\|=1$ shows the first inequality. For the second 
inequality, let $v\in V\cap S_X$. If $1-\frac\delta2<\|w\|<1+\frac\delta 2$,
then $\|v+w\|\ge \|v+\frac{\|v\|}{\|w\|}w\|-|\|v\|-\|w\||\ge \frac\delta 2$.
If $\|w\|$ lies outside this range, then the same conclusion follows from the 
triangle inequality, so that $\|\proj_{V\parallel W}(v+w)\|=\|v\|
\le \frac 2\delta\|v+w\|$. 
\eqref{it:projbound2} can be found in \cite{Wojtaszczyk}, Corollary III.B.11.
\end{proof}

For a Banach space $X$, the bounded linear maps from $X$ to itself will be 
written $B(X,X)$ and $\mathcal B_X$ will be the Borel $\sigma$-algebra on $X$.
In this section, we consider \textsl{random dynamical systems}.  These
consist of a tuple $\mathcal{R}=(\Omega,\mathcal F,\mathbb P, \sig, X,
\LL)$, where $(\Omega,\mathcal F,\mathbb P)$ is a complete
probability space; $\sig$ is a measure
preserving transformation of $\Omega$; $X$ is a separable Banach space; the
\textsl{generator} $\LL\colon \Omega\to B(X,X)$
is strongly measurable (that is for fixed $x\in X$,
$\omega\mapsto \LL_\omega x$ is $(\mathcal F,\mathcal B_X)$-measurable); and 
$\log\|\LL_\om\|$ is integrable. An alternative description of strong
measurability is that the map $\omega\mapsto \LL_\omega$ is $(\mathcal F,
\mathcal S)$-measurable, where $\mathcal S$ is the Borel $\sigma$-algebra
of the strong operator topology on $B(X,X)$ (see Appendix A of \cite{GTQuas} for details). 
In the context where $X$ is separable and the operators are bounded, 
strong measurability is equivalent to 
$(\mathcal F\otimes \mathcal B_X,\mathcal B_X)$-measurability of  
the map $(\omega,x)\mapsto
\LL_\omega x$ (\cite{GTQuas}).

A random dynamical system gives rise to a cocycle of bounded linear
operators $\LL_\om^{(n)}$ on $X$, defined
by $\LL_\om^{(n)}(x)=\LL_{\sigma^{n-1}\om}\circ\dots\circ\LL_\om x$.
We will consider $\mathcal{F}$ and
$\mathbb P$ to be fixed, and thus refer to a random dynamical system
as $\mathcal{R}=(\Omega,\sigma,X,\mathcal L)$.  We say $\mc{R}$ is
ergodic if $\sig$ is ergodic.

When the base $\sig$ is invertible, we can also define the dual random
dynamical system $\mc{R}^*=(\Omega,\mathcal F,\mathbb P, \sig^{-1},
X^*, \LL^*)$, where $X^*$ is the dual of $X$ and $\LL^*_\om(\theta)
:= (\LL_{\sig^{-1}\om})^* \theta$. Notice that $\LL^*_\om$ is \emph{not}
$(\LL_\om)^*$. The rationale for this
is that $\LL_\om$ maps the $X$-fibre over $\omega$ to the $X$-fibre over
$\sigma(\omega)$ and similarly $\LL_\om^*$ maps the $X^*$-fibre over $\omega$
to the $X^*$-fibre over $\sigma^{-1}\omega$.
In this way, $\theta(\LL_\om x)=
\LL^*_{\sig\om}\theta(x)$ and, more generally,
$\LL^{*(n)}_{\sigma^n\omega}\theta(x)=\theta(\LL^{(n)}_\omega x)$ for
every $x\in X, \theta \in X^*$.  Thus, $\LL^{*(n)}_{\sig^n \om}=
(\LL_{\om}^{(n)})^*$.

\begin{lem}[Measurable dense subset of a family of subspaces]
\label{lem:measurabledense}
Let $X$ be a separable Banach space. Let $V\colon \Omega\to 
\mathcal G^k(X)$ be measurable. Then there exist sequences of
measurable functions $u_n\colon\Omega\to S_X$ 
and $u_n'\colon \Omega\to B_X$ such that
$\{u_n(\om)\colon n\in\N\}$ is a dense subset of $V(\om)\cap S_X$
and $\{u_n'(\om)\colon n\in \N\}$ is a dense subset of $V(\om)\cap B_X$.
\end{lem}

\begin{proof}
  First, for fixed $v\in X, \om \mapsto d(v,V(\om))$ is a measurable
  function, as it is the composition of continuous and measurable
  functions.  Fix a dense sequence $v_1,v_2, \ldots \in S_X$.  Now for
  each $j$, set $u_j^0(\om) = v_j$, and let $u_j^{k+1}(\om) = v_l$,
  where $l=\min \{ m: d(v_m,V(\om)\cap S_X) \leq \tfrac 14
  d(u_j^{k}(\om),V(\om)\cap S_X) \text{ and } d(v_m,u_j^k(\om))\leq 2
  d(u_j^k(\om),V(\om)\cap S_X)\}$.  For each $j$, this is a measurable
  convergent sequence and hence the limit point $u_j^\infty(\om)$ is
  measurable, and belongs to $V(\om)\cap S_X$.  The sequence
  $(u_j^\infty(\om))$ is dense in $V(\om) \cap S_X$ because there are
  $v_j$ arbitrarily close to all points of $V(\om)\cap S_X$. The functions
  $u_n'$ are produced exactly analogously.
\end{proof}

\begin{lem}[Measurability of growth measurements]\label{lem:volmeas}
Let $\mathcal R$ be a random dynamical system
$\mathcal{R}=(\Omega,\sigma,X,\mathcal L)$ acting on a separable
Banach space. 
The following functions are measurable:
\begin{itemize}
\item $\omega\mapsto D_k(\LL_\omega)$;
\item $\omega\mapsto \|\LL_\omega\|$;
\item $\omega\mapsto \ic({\LL_\omega}):=
\inf\Big\{$\parbox{2.5in}{$r>0 : \LL_\om(B_X)$ can be covered by
finitely many balls of radius $r$} $\Big\}$.
\end{itemize}
Further, if $V\colon \Omega\to \mathcal G^k(X)$ is measurable, then 
$\omega\mapsto \|\LL_\om|_{V(\om)}\|$ is measurable. 
\end{lem}

\begin{proof}
Let $(x_n)$ be a dense subsequence of $B_X$.
By strong measurability, for each fixed $n$, $\omega\mapsto
\|\LL_\omega x_n\|$ is measurable.  Then for each $j_1,\ldots,j_i$, we
have that $f_{j_i|j_1,\ldots,j_{i-1}}(\omega) :=\inf_{q_1,
\dots,q_{i-1}\in\mathbb Q} \|\LL_\omega x_{j_i}-\sum_{1\le l<i}
q_l \LL_\omega x_{j_l}\|$ is measurable, so
$$
D_k(\LL_\omega)= \sup_{j_1,\ldots,j_k}\prod_{i\le
  k}f_{j_i|j_1,\ldots,j_{i-1}}(\omega)
$$
is measurable. In particular, $\omega\mapsto\|\LL_\omega\|=D_1(\LL_\om)$ is measurable. 
We claim that 
\begin{equation}\label{eq:kappa}
\ic(\LL)=\lim_{n\to\infty}
\sup_j \inf_{k\le n}\Big\|\LL x_j-2\|\LL\|x_k\Big\|.
\end{equation}
If this limit is $r$, then there exists $n$ such that
$\sup_j\inf_{k\le n}\Big\|\LL x_j-2\|\LL_\om\|x_k\Big\|<r+\epsilon$. 
This gives a covering of $\{\LL x_j\colon j\in\N\}$ by $n$
balls of radius $r+\epsilon$, so that the left side of \eqref{eq:kappa} is dominated
by the right side. Conversely, if $\ic(\LL)=r$, let $\LL(B_X)$ be covered by 
finitely many balls of radius $r+\epsilon$. These must have centres with norm
at most $2\|\LL\|$ otherwise they do not intersect $\LL(B_X)$. The centres must
therefore be $\epsilon$-approximable by points of the form $2\|\LL\|x_k$, so that
the right side of \eqref{eq:kappa} is at most $r+2\epsilon$. We deduce 
\eqref{eq:kappa} holds and $\omega\mapsto \ic(\LL_\om)$ is measurable. 

Finally, if $V\colon \Omega\to \mathcal G^k(X)$ is measurable, let 
$(u_n(\om))_{n\in\N}$ be a sequence of measurable functions such that
$\{u_n(\om)\colon n\in \N\}$ is dense in $S_{V(\om)}$. Then
$\|\LL_\om|_{V(\om)}\|=\sup_n\|\LL_\om u_n(\om)\|$, which is therefore
measurable. 
\end{proof}

When $\mc{R}$ is ergodic, Lemma~\ref{lem:volmeas} 
combined with Kingman's sub-additive
ergodic theorem ensures the existence of the \textsl{maximal Lyapunov
  exponent} of $\mc{R}$, defined by
$$
\lam(\mc{R}):= \lim_{n\to\infty}\tfrac1n\log \|\LL^{(n)}_\omega\|,
$$
for $\paeom$.  Similarly, using the fact that the Kuratowski index of compactness,
$\kappa(\LL)$,
is also sub-multiplicative and bounded above by the norm, we have
existence of the \textsl{index of compactness} of $\mc{R}$, defined by
$$
\ka(\mc{R}):= \lim_{n\to\infty}\tfrac1n\log \ic(\LL^{(n)}_\omega),
$$
with the property that $\ka(\mc{R}) \leq \lam(\mc{R})$.  

In the case where 
$\LL_\omega$ is independent of $\omega$, $\lam(\mc{R})$ and $\ka(\mc R)$
are the spectral radius and essential spectral radius respectively,
so that $\ka(\mc R)<\lam(\mc R)$ is the quasi-compact case. If the operator
is compact, then $\ka(\mc R)$ is 0. 

Our previous paper \cite{GTQuas} studies the case
in which $\mc R$ is a random dynamical system
where the operators $\LL_\om$ are Perron-Frobenius operators of a family of 
expanding maps and gives sufficient conditions for $\ka(\mc R)<\lam(\mc R)$.

\begin{lem}\label{lem:expDefn}
  Given an ergodic random dynamical system $\mc{R}$, there exist
  constants $\Delta_k=\Delta_k(\mc{R})$ such that for almost every
  $\omega\in\Omega$,
$$
\lim_{n\to\infty}\tfrac1n\log D_k(\LL^{(n)}_\omega)=\Delta_k.
$$
Furthermore, $\frac 1n\log E_k(\LL^{(n)}_\omega)\to\Delta_k$.  Define
$\Delta_0=0$ and let $\mu_k=\Delta_k-\Delta_{k-1}$
for each $k\ge 1$.  Then, $\frac 1n\log
F_k(\LL^{(n)}_\omega)\to\mu_k$.
\end{lem}

\begin{proof}
  The first claim follows from Kingman's sub-additive ergodic theorem,
  via Lemma~\ref{lem:volmeas} and
  Lemma~\ref{lem:submult}.  The remaining two claims are consequences
  of Corollary \ref{cor:equiv}.
\end{proof}

The $\mu_k$'s of the previous lemma are called the \textsl{Lyapunov
  exponents} of $\mc{R}$. When $\mu_k>\kappa(\mc{R})$, $\mu_k$ is
called an \textsl{exceptional Lyapunov exponent}.

\begin{thm}[Lyapunov exponents and index of compactness]\label{thm:summ}
  Let $\mc{R}$ be a random dynamical system with ergodic base acting
  on a separable Banach space $X$.  Then
  \begin{itemize}
  \item $\mu_1\ge \mu_2\ge\ldots$;
  \item For any $\rho>\kappa(\mathcal R)$, there are only finitely
    many exponents that exceed $\rho$;
  \item If $\sig$ is invertible, then $\mc{R}$ and $\mc{R}^*$
   have the same Lyapunov exponents.
  \end{itemize}
\end{thm}

\begin{proof}
  That the $\mu_i$ are
  decreasing follows from Lemma \ref{lem:expDefn} and
  the observation that $F_k(T)\le F_{k-1}(T)$.
  That the system and its dual have the same exponents follows from
  Lemma \ref{lem:primaldual} together with the simple result (in
  \cite[Lemma 8.2]{FLQ1}) that if $(f_n)$ is sub-additive and
  satisfies $f_n(\omega)/n\to A$ almost everywhere, then one has
  $f_n(\sigma^{-n}\omega)/n\to A$ also.

  It remains to show that for $\rho>\kappa$, the system has at most
  finitely many exponents that exceed $\rho$.  Let
  $\kappa<\alpha<\beta<\rho$. Since $\log\|\LL_\omega\|$ is
  integrable, there exists a $0<\delta<(\beta-\alpha)/2|\alpha|$ such
  that if $\mathbb P(E)<\delta$, then $\int_E
  \log^+\|\LL_\omega\|\,d\PP(\omega)<(\beta-\alpha)/2$.  By the
  sub-additive ergodic theorem, there exists $L>0$ such that
  $\PP(\ic(\LL^{(L)}_\omega)\ge e^{\alpha L})<\delta/2$.  If
  $\ic(\LL^{(L)}_\omega)< e^{\alpha L}$, then by definition,
  $\LL^{(L)}_\omega B_X$ may be covered by finitely many balls of size
  $e^{\alpha L}$. By linearity, if $\ic(A)=\zeta$, 
  one sees that if $B$ is a ball with arbitrary centre and radius $\rho$, then $A(B)$ 
  may be covered by finitely many balls of size $\zeta \rho$.
  
  Let $r$ be chosen large enough so that
  $\PP(G)>1-\delta$, where $G$ (the good set) is defined by
$$
G=\{\omega:\LL^{(L)}_\omega B_X\text{ may be covered with $e^{rL}$
  balls of size $e^{\alpha L}$}\}.
$$
We split the orbit of $\omega$ into blocks: if $\sigma^i\omega\in G$,
then the block length is $L$; otherwise, if $\sigma^i\omega$ is bad,
we take a block of length $1$. Consider the following iterative process:
start with a ball of radius $\rho_0=1$. Then look at the current iterate of
$\omega$, $\sigma^i\omega$, and suppose that $\LL^{(i)}_\om B_X$ is covered 
by $N_i$ balls of radius $\rho_i$. If $\sigma^i\om\in G$, then $\LL^{(i+L)}_\om
B_X$ is covered by at most $N_{i+L}=N_ie^{rL}$ balls of radius $\rho_{i+L}=
e^{\alpha L}\rho_i$ and the new iterate is $\sigma^{i+L}\om$. 
If $\sigma^i\om\not\in G$,
then $\LL^{(i+1)}_\om B_X$ is covered by at most $N_{i+1}=N_i$ balls of radius 
$\rho_{i+1}=\|\LL_{\sigma^i\om}\|\rho_i$ and the new iterate is $\sig^{i+1}\om$.

We claim that for almost all $\omega$, for sufficiently large $N$,
$\LL^{(N)}_\omega(B_X)$ is covered by at most $e^{rN}$ balls of size
$e^{\beta N}$. Indeed, given $\omega$, let $n_0$ be chosen such that 
for all $N\ge n_0$, one has $\sum_{i=0}^{N-1}
\mathbf 1_{G^c}(\sig^i\om)\log^+\|\LL_{\sig^i\om}\|<(\beta-\alpha)N/2$. 
If $\alpha\ge 0$, then for large $N$, through
the good steps, the balls are inflated by a factor at most $e^{\alpha
  N}$. If $\alpha<0$, then combining the good
blocks, the balls are scaled by a factor of $e^{\alpha(1-\delta)N}
<e^{(\alpha+\beta)N/2}$ or
smaller. In both cases, we see that overall, balls are 
scaled by at most $e^{\beta N}$. The splitting only takes place
in the good blocks, and yields at most $e^{rN}$ balls.

Now suppose that $\mu_k> \rho$. For almost all $\omega$, we have that
for all large $N$, $D_k(\LL_\omega^{(N)})>e^{kN\rho}$.  Fix
such an $N$, and suppose that $x_1,\ldots,x_k$ belong to $S_X$ 
and have the property
$\prod_{i\le k}D_i>e^{kN\rho}$ where
$$
D_i=d(\LL^{(N)}_\omega x_i,\lin(\{\LL^{(N)}_\omega x_j\colon j<i\})).
$$

Let $T_i=\{0,1,\ldots,\lfloor D_i/(2ke^{\beta N})\rfloor\}$ and notice
that $|T_1\times \cdots\times T_k|\ge e^{kN(\rho-\beta)}/(2k)^k$.  For
$(j_1,\ldots,j_k)\in T_1\times\cdots\times T_k$, define
$$
y_{j_1,\ldots,j_k}=\sum_{i=1}^k \frac {2j_ie^{\beta
    N}}{D_i}\LL^{(N)}_\omega x_i.
$$
It is not hard to see that all of these points belong to the image of
the unit ball of $X$ under $\LL^{(N)}_\omega$. Further, from the
definition of $D_i$, one can check that these points are mutually
separated by at least $2e^{\beta N}$, so that one requires at least
$e^{kN(\rho-\beta)}/(2k)^k$ balls to cover $\LL^{(N)}_\omega(B_X)$.

Hence we obtain
$$
\frac {e^{kN(\rho-\beta)}} {(2k)^k}\le e^{rN}.
$$
Since this holds for all large $N$, we deduce $k\le r/(\rho-\beta)$ as
required.
\end{proof}

\begin{lem}[Measurability II]\label{lem:grassmeas}
  Suppose that $X$ is a Banach space with separable dual.
  Suppose further that $\mathcal R$ is an ergodic random dynamical
  system acting on $X$.

  Assume there exist $\lambda'>\lambda\in\R$ and $d\in\mathbb N$ such that
  for $\PP$-almost every $\omega$, there is a closed $d$-codimensional
  subspace $V(\omega)$ of $X$ such that:
  \begin{enumerate}
  \item for all $v\in V(\omega)$,
  $\limsup_{n\to\infty} \frac 1n\log\|\mathcal L^{(n)}_\omega
  v\|\le\lambda$; and
  \item for each $a>0$ and $\ep>0$, there is an $n_0$ such that
  for $v\in S_X$ satisfying $d(v,V(\omega))>a$, one has
    $\|\mathcal L^{(n)}_\omega v\|\ge e^{n(\lambda'-\ep)}$
    for all $n\ge n_0$.
  \end{enumerate}
   Then $\omega\mapsto V(\omega)$ is measurable.
\end{lem}

\begin{proof}
  Given $V\in \mathcal G^d(X)$, fix  $w_1,\ldots,w_d$ such that
  $V\oplus\lin(w_1,\ldots,w_d)=X$ and define a neighbourhood of $V$ by
\begin{align*}
N_{V,k}= \{&U\in \mathcal G^d(X)\colon
U\cap \lin(w_1,\ldots,w_d)=\{0\};\\
&\|\proj_{\lin(w_i)\parallel U\oplus \lin(\{w_j\colon j\ne
  i\})}|_V\|\le \tfrac1k\text{ for $1\le i\le d$}\}.
\end{align*}

Since $\mathcal G^d(X)$ is separable, fix a countable
sequence $(V_n)$ of subspaces, dense in $\mathcal G^d(X)$.
Using Lemma \ref{lem:Grass} items \eqref{it:projbound2} and
\eqref{it:symmclose}, there exists $K>0$ such that
for each $V\in \mathcal G^d(X)$, there exists $W\in \mathcal G_d(X)$
such that $\|\proj_{W\parallel V}\|\le K$,
$\|\proj_{V\parallel W}\|\le K$. If $w_1,\ldots,w_d$ is an Auerbach
basis for $W$, then $\|\proj_{\lin(w_{i})
\parallel\lin(\{w_{j}\colon j\ne i\})}|_{W}\|\le 1$.
For each $V_n$, let $W_n$ be a subspace satisfying the above inequalities
and let $w_{n,1},\ldots,w_{n,d}$ be an Auerbach basis. 
Let $P_{n,i}$ denote $\proj_{\lin(w_{n,i})\parallel V(\om)\oplus 
\lin(\{w_{n,j}\colon j\ne i\})}$.

We obtain a countable collection of basic sets $(N_{V_n,k})$
which generate the Borel $\sigma$-algebra on $\mathcal G^d(X)$. To see
this, we claim that for each $V\in \mathcal{G}^d(X)$ and each open set $O$ containing 
$V$, there are $n$ and $k$ such that $V\in N_{V_n,k}\subset O$. 
Then each open set is the union of the basic sets that it contains.

Given $V\in \mathcal G^d(X)$ and an open set $O$ containing it, 
let $B_r(V)\subset O$. Let $k>4Kd/r$ and $\delta=\min(1/(2K),1/(4kK),r/2)$. 
Let $n$ be such that $d_H(V\cap S_X,V_n\cap S_X)<\delta$. 
Let $v\in V\cap S_X$ and $w\in W_n\cap S_X$. There exists $v'\in V_n\cap S_X$
such that $\|v-v'\|<\delta$. By Lemma \ref{lem:Grass}\eqref{it:projbound},
$\|w-v'\|\ge 1/K$, so that $\|w-v\|\ge 
1/(2K)$.
Hence $V\cap W_n=\{0\}$ and $\|\proj_{W_n\parallel V}\|\le 4K$. 
Now given $v'\in V_n\cap S_X$, there exists $v\in V\cap S_X$ and $x\in X$ with 
$v'=v+x$ and $\|x\|<\delta$. We have 
$\|P_{n,i}(v')\| =\|P_{n,i}(x)\| =
\|P_{n,i}\circ \proj_{W_n\parallel V}(x)\|\le 4K\delta$, 
so that $V\in N_{V_n,k}$. Finally let $U\in N_{V_n,k}$ and
let $v'\in V_n\cap S_X$. By definition, we have
$\|P_{n,i}(v')\|\le \tfrac 1k$ for each $i$, so that
$\|\proj_{W_n\parallel U}(v')\|\le \tfrac dk$. In particular, there
exists $u\in U$ such that $\|u-v'\|\le \tfrac dk$ and hence there is
$u\in U\cap S_X$ such that $\|u-v'\|\le \tfrac {2d}k$. Using Lemma 
\ref{lem:Grass}\eqref{it:symmclose}, we
deduce $d_H(U\cap S_X,V_n\cap S_X)\le \frac{2Kd}k$, so that
$d_H(U\cap S_X,V\cap S_X)\le \delta+\frac{2Kd}k<r$, showing
$N_{V_n,k}\subset O$. 

Hence to show the desired measurability, it suffices to show that for
each $N =N_{V_n,k}$, $\{\omega\colon V(\omega)\in
N\}$ is measurable. First, $\{U\colon U\cap W_n=\{0\}\}$
is an open set, so that $\{\omega\colon V(\omega)\cap W_n=\{0\}\}$
is measurable. 

Fix a dense set $v_1,v_2,\ldots$ in the unit sphere of
$V_n$.  We claim that for those $\omega$ lying in the set $G$ of full measure
on which $\dim V(\omega)=d$ and hypotheses (1) and (2) of the Lemma hold,
we have that $V(\omega)$ lies in $N$ if and only if the following condition holds:

\begin{quote}
For each rational $\epsilon>0$ and each $j\in \N$, there is $m_0>0$ such that for
each $m\ge m_0$, there are rationals
$a^j_1,\ldots,a^j_d$ in $[-\frac1k,\frac 1k]$ such that $
\|\mathcal L_\omega^{(m)}(v_j-\sum_{i=1}^da^j_iw_{n,i})\| \le
e^{(\lambda+\epsilon)m}$.
\end{quote}

To see the `only if' direction, suppose that $V(\om)\in N$. 
Now given $v_j\in V_n\cap S_X$, by definition of $N$, there are $b_1,\ldots,b_d$
in $[-\frac1k,\frac 1k]$ such that  
$v_j=v'+b_1w_{n,1}+\ldots+b_dw_{n,d}$ with $v'\in V(\om)$. Hence
$v'=v_j-\sum_{i=1}^d b_i w_{n,i}\in V(\om)$, and therefore
we have $\|\LL_\om^{(m)}v'\|< e^{(\lambda+\epsilon)m}$ for all
sufficiently large $m$. Now for any such $m$, one can take
$a_i$'s that are suitably close rational approximations to $b_i$
so that $\|\LL_\om^{(m)}(v_j-\sum_{i=1}^d a_iw_{n,i})\|\le e^{(\lambda+\epsilon)m}$. 

Conversely, suppose that $V(\om)\cap \lin W_n=\{0\}$, but 
$V(\om)\not\in N$. Then there exists a $v\in V_n\cap S_X$ and an $i$
such that $\|P_{n,i}(v)\|
>\frac 1k$. By continuity, there exists a $v_j$ satisfying the same property. 
Let $\delta=\|P_{n,i}(v_j)\|-\frac 1k$. Then 
$\|P_{n,i}(v-\sum_{l=1}^d
a_lw_l)\|\ge\delta$ for all $(a_l)_{l=1}^d\in [-\tfrac 1k,\tfrac1k]^d$.
By hypothesis, we now see that the condition is not satisfied. 


Since this condition is obtained by taking countable unions and
intersections of measurable sets, the measurability of $G\cap \{\omega\colon
V(\omega)\in N\}$ is demonstrated. Using completeness of $(\Om,\mathcal F,\mathbb P)$,
we deduce that $\{\om\colon V(\om)\in N\}$ is measurable, so that $\om\mapsto
V(\om)$ is measurable as required. 
\end{proof}

\begin{lem}\label{lem:tempered}
  Let $\sig$ be an ergodic measure-preserving transformation of a
  probability space $(\Om,\PP)$. Let $g$ be a non-negative measurable
  function and let $h\geq0$ be integrable.  Suppose further that
  $g(\om)\le h(\om)+g(\sig(\om))$, $\PP$-a.e.  Then $g$ is tempered;
  that is $\lim_{n\to\infty}g(\sig^n\om)/n=0$, $\PP$-a.e.
\end{lem}

A proof of this lemma appears in Ma\~n\'e's paper \cite{Mane}. 

\begin{proof}
  Let $\epsilon>0$ and let $K>\int h$. By the maximal ergodic
  theorem, $B_1=\{\om\colon h(\om)+\ldots+h(\sigma^{n-1}\om)<nK,
  \text{ for all $n$}\}$ has positive measure. 
  Let $M$ be such that $B_2:=\{\om\colon g(\om)<M\}$
  has positive measure.
  As a consequence of the
  Birkhoff ergodic theorem, for any measurable set $B$ with
  $\PP(B)>0$, for $\PP$-a.e. $\om$, for all sufficiently large $k$,
  there exists $j\in [(1+\ep)^k,(1+\ep)^{k+1})$ such that
  $\sig^j(\om)\in B$. Now for $\om\in\Om$, let $k_0$ be such that for
  all $k\ge k_0$, there exist $j\in [(1+\ep)^k,(1+\ep)^{k+1})$ such

  that $\sig^j\om\in B_1$ and $j'\in((1+\ep)^{k+2},(1+\ep)^{k+3})$
  such that $\sig^{j'}\om\in B_2$.  If $n>(1+\ep)^{k_0+1}$, then $n\in
  [(1+\ep)^{k+1},(1+\ep)^{k+2})$ for some $k\ge k_0$.
  Let $j\in
  [(1+\ep)^k,(1+\ep)^{k+1})$
 and $j'\in [(1+\ep)^{k+2},(1+\ep)^{k+3})$
  be as above. Then $g(\sig^n\om)\le \sum_{k=n}^{j'-1}h(\sigma^k\om)
   + g(\sig^{j'}\om) \le \sum_{k=j}^{j'-1}h(\sig^k\om)
  +g(\sig^{j'}\om)\le K(j'-j)+M$, so that $\limsup
  g(\sig^n\om)/n\le 4\ep$. Since $\ep$ is arbitrary, the conclusion
  follows.
\end{proof}

\begin{thm}[Multiplicative ergodic theorem: the Oseledets filtration]
\label{thm:Osel}
Let $\mathcal R$ be an ergodic random dynamical system acting on a
Banach space $X$ with separable dual. Suppose that
$\kappa(\mathcal R)<\lambda(\mathcal R)$. Then there exist $1\le
r\le\infty$
  \footnote{If $r=\infty$, the conclusions are replaced
    by: $\lambda(\mathcal R)=\lambda_1>\lambda_2>\ldots\to \kappa(\mathcal R)$,
    $\mult_1,\mult_2,\ldots\in \N$ and $X=V_1(\om)\supset V_2(\om)
    \supset \dots$; $V_\infty(\om)=\bigcap V_i(\om)$.}  and:
  \begin{itemize}
  \item a sequence of exceptional Lyapunov exponents $\lambda(\mathcal
    R)=\lambda_1>\lambda_2>\ldots>\lambda_r>\kappa(\mathcal R)$;
  \item a sequence $\mult_1,\mult_2,\ldots,\mult_r$ of positive
    integers; and
  \item a measurable filtration of closed subspaces, $X=V_1(\om)\supset
    V_2(\om) \supset \dots \supset V_r(\omega)\supset V_\infty(\om)$,
    with the equivariance property $\LL_\om(V_i(\om))\subset V_i(\sig(\om))$
    for each $i$. 
\end{itemize}
such that for $\PP$-a.e.\ $\omega$,  $\codim V_\ell(\omega)=\mult_1+\dots+\mult_{\ell-1}$;
for all $v\in V_\ell(\omega)\setminus V_{\ell+1}(\omega)$, one has
$\lim\frac 1n\log\|\LL^{(n)}_\omega v\|= \lambda_\ell$; and for 
$v\in V_\infty(\omega)$, $\limsup\frac 1n\log\|\LL^{(n)}_\omega v\|\le \kappa(\mathcal R)$.
%
\end{thm}

While the theorem is stated for ergodic random dynamical systems, a standard
application of ergodic decomposition allows one to deduce a version for non-ergodic
systems, in which constants are replaced by invariant functions. 

\begin{proof}
  Let $\mu_1\ge \mu_2\ge\ldots$
  be as in Lemma \ref{lem:expDefn}. Let
  $\lambda_1>\lambda_2>\ldots$ be the decreasing enumeration of the
  distinct $\mu$-values that exceed $\kappa(\mathcal R)$ (if this an
  infinite sequence, then Theorem \ref{thm:summ} establishes that
  $\lambda_i\to\kappa(\mathcal R)$).  The fact that $\lambda(\mathcal
  R)=\lambda_1$ is straightforward from the definitions.  Let
  $\mult_\ell$ be the number of times that $\lambda_\ell$ occurs in
  the sequence $(\mu_i)$ and let $\D_\ell=\mult_1+\ldots +\mult_\ell$,
  so that $\mu_{\D_{\ell-1}}=\lambda_{\ell-1}$ and
  $\mu_{\D_{\ell-1}+1}=\lambda_{\ell}$.

  We now turn to the construction of $V_\ell(\omega)$.  For a fixed
  $\omega$, let the sequences $(\theta^{(n)}_i)_{i\ge 1}$ and
  $(x^{(n)}_j)_{j\ge 1}$ be as guaranteed by Lemma \ref{lem:genseqs}
  for the operator $\mathcal L^{(n)}_\omega$.  We let
  $V_\ell^{(n)}(\omega)$ be
  $\lin((\LL^{(n)}_{\om})^*\theta^{(n)}_1,\ldots,
  (\LL^{(n)}_{\om})^*\theta^{(n)}_{\D_{\ell-1}})^\perp$, let
  $Y_\ell^{(n)}(\omega)$ be $\lin(\{x^{(n)}_j\colon j\le
  \D_\ell\})$. Thus, $X=V_\ell^{(n)}(\omega)\oplus
  Y_{\ell-1}^{(n)}(\omega)$.
  All of these depend on the choice of $\theta$'s and $x$'s.  No claim
  of uniqueness or measurability is made. 
  
  The space $V_\ell^{(n)}(\om)$ is
  an approximate slow space.

  The proof will go by the following steps:
  \begin{enumerate}[(a)]
  \item \label{it:nrate} For almost all $\omega$, for arbitrary $\epsilon>0$
    and for sufficiently large $n$, $\|\mathcal
    L^{(n)}_\omega x\|\le e^{(\lambda_\ell+\epsilon)n}\|x\|$ for all
    $x\in V_\ell^{(n)}(\omega)$;
  \item \label{it:Cauchy} $V_\ell^{(n)}(\omega)$ is a Cauchy sequence
    for almost all $\omega$ -- we define the limit to be
    $V_\ell(\omega)$;
  \item \label{it:equivar} The $V_\ell(\omega)$ are equivariant: $\mathcal L_\omega(V_\ell(\omega))\subseteq
    V_\ell(\sigma(\omega))$;
  \item \label{it:liminfrate} If $x\not\in V_{\ell+1}(\omega)$, then
  $\| \mathcal L^{(n)}_{\om} v \| > e^{(\lambda_\ell-\ep)n} d(v,V_{l+1}(\om))$
  for large $n$; 
\item \label{it:unifliminf} For all $a>0$ and $\epsilon>0$, there exists $n_0$ so that
for all $n\ge n_0$ and all $x\in S_X$ such that $d(x,V_{\ell+1}(\om))\ge a$, one has
$\|\LL_\om^{(n)}x\|\ge e^{(\lambda_\ell-\ep)n}$.
    \setcounter{saveenum}{\value{enumi}}
  \end{enumerate}
  The remaining steps are proved by induction on $\ell$.
  \begin{enumerate}[(a)]
    \setcounter{enumi}{\value{saveenum}}
  \item \label{it:limsuprate} If $x\in V_\ell(\omega)$, then
    $\limsup\frac 1n\log\|\mathcal L^{(n)}_\omega x\|\le
    \lambda_\ell$;
  \item \label{it:mble}$\omega\mapsto V_\ell(\omega)$ is measurable;
  \item \label{it:inducexpts} The restriction, $\mathcal R_\ell$, of
    $\mathcal R$ to $V_\ell(\omega)$ has the same exponents as
    $\mathcal R$ with the initial $\D_{\ell-1}$ exponents removed.
  \end{enumerate}

\paragraph{\emph{Proof of} \eqref{it:nrate}}
Note that by construction
$$
\det\left( (\theta^{(n)}_i(\mathcal L^{(n)}_\omega x^{(n)}_j))_{1\le
    i,j\le \D_{\ell-1}} \right) \ge
KE_{\D_{\ell-1}}(\mathcal L_\omega^{(n)}),
$$
where $K$ is a constant depending only on $\D_{\ell-1}$ arising from Lemmas
\ref{lem:Frec} and \ref{lem:genseqs}.
For an arbitrary $x\in V_\ell^{(n)}\cap S_X$, let $\phi \in S_{X^*}$
be such that $\phi(\mathcal L^{(n)}_\omega x)=\|\mathcal
L^{(n)}_\omega x\|$.  Then, adding a column for $x$ and a row for
$\phi$ to the matrix $U((\theta_i^{(n)}),(x_j^{(n)}))_{1\le i,j\le
  \D_{\ell-1}}$, we see that the $x$ column has all 0 entries except
for the $1+\D_{\ell-1}$-st (by definition of $V^{(n)}_\ell(\omega)$),
and so we arrive at the bound (uniform over $x\in S_X \cap
V_\ell^{(n)})$,
\begin{equation}\label{eq:unifexprate}
 KE_{\D_{\ell-1}}(\mathcal L^{(n)}_\omega)
  \|\mathcal L^{(n)}_\omega x\|
  \le E_{1+\D_{\ell-1}}(\mathcal L^{(n)}_\omega).
\end{equation}
The conclusion follows from Lemma~\ref{lem:expDefn}.

\paragraph{\emph{Proof of} \eqref{it:Cauchy}}
Let us assume that $n_0$ is chosen large enough that for all $n\ge
n_0$, the following conditions are satisfied: $\|\mathcal L_\omega\|$,
$\|\mathcal L_{\sigma^n\omega}\|$ are less than $e^{\epsilon n}$;
$\|\mathcal L^{(n)}_\omega x\|\leq e^{(\lambda_\ell+\epsilon)n}\|x\|$
for all $x\in V_\ell^{(n)}$; and $\|\mathcal L_\om^{(n)}x\|\geq
e^{(\lambda_{\ell-1}-\epsilon)n}\|x\|$ for all $x\in Y_{\ell-1}^{(n)}(\om)$
(using integrability of $\log\|\LL_\om\|$; \eqref{it:nrate}; and Lemma \ref{lem:genseqs}).
Let $n\ge n_0$.  Let $x\in
V_\ell^{(n)}(\om)\cap S_X$ and write $x=u+w$ where $u\in
V_\ell^{(n+1)}(\om)$ and $w\in Y^{(n+1)}_{\ell-1}(\om)$.  Now we have
$$
\|\mathcal L_\om^{(n+1)}x\|\le e^{(\lambda_\ell+\epsilon)n} \|\mathcal
L_{\sigma^n\omega}\| \le e^{(\lambda_\ell+2\epsilon)n}.
$$
We also have $\|u\|\le 1+\|w\|$, $\|\mathcal L^{(n+1)}_\omega w\|\ge
e^{(\lambda_{\ell-1}-\epsilon)(n+1)}\|w\|$ and $\|\mathcal
L^{(n+1)}_\omega u\|\le e^{(\lambda_\ell+\epsilon)(n+1)}(1+\|w\|)$.
Manipulation with the triangle inequality yields
\begin{equation}\label{eq:wbound}
  \|w\|\le e^{-n(\lambda_{\ell-1}-\lambda_\ell-4\epsilon)}.
\end{equation}
Hence, each point in the unit sphere of $V_\ell^{(n)}(\omega)$ is
exponentially close to $V_\ell^{(n+1)}(\omega)$. Since the two spaces
have the same codimension, one obtains a similar inequality in the
opposite direction by Lemma \ref{lem:Grass}\eqref{it:symmclose}.  
This establishes that $V_\ell^{(n)}(\omega)$ is a Cauchy sequence.

\paragraph{\emph{Proof of} \eqref{it:equivar}}
We argue essentially as in \eqref{it:Cauchy}. For
large $n$, we take $v\in V^{(n+1)}_\ell(\omega)\cap S_X$. We write
$\mathcal L_\omega(v)$ as $u+w$ with $u\in
V^{(n)}_\ell(\sigma(\omega))$ and $w\in
Y^{(n)}_{\ell-1}(\sigma(\omega))$.  We have
bounds of the form $\|\mathcal L^{(n+1)}_\omega v\|\lesssim
e^{\lambda_\ell n}$; $\|u\|\lesssim 1+\|w\|$, $\|\mathcal
L^{(n)}_{\sigma(\omega)}u\|\lesssim e^{\lambda_\ell n}(1+\|w\|)$ and
$\|\mathcal L^{(n)}_{\sigma(\omega)}w\|\gtrsim
e^{\lambda_{\ell-1}n}\|w\|$ (here $\lesssim$ means `is smaller up to
sub-exponential factors').  Combining the inequalities as before, one
obtains a bound $\|w\|\lesssim e^{-(\lambda_{\ell-1}-\lambda_\ell)n}$.
Taking a limit, we obtain $\mathcal L_\omega V_\ell(\omega)\subset
V_\ell(\sigma(\omega))$ as required.

\paragraph{\emph{Proof of} \eqref{it:liminfrate}}
Let $x\not\in V_{\ell+1}(\omega)$, with $\|x\|=1$.
For large $n$, if $x$ is written as
$u_n+v_n$ with $u_n\in V^{(n)}_{\ell+1}(\omega)$ and
$v_n\in Y^{(n)}_\ell(\om)$, then $\|v_n\|\ge
\frac 12d(x,V_{\ell+1}(\om))$ and $\|u_n\|\le 1+\|v_n\|$. By \eqref{eq:unifexprate},
$\|\mathcal L^{(n)}_\omega u_n\|\le
e^{(\lambda_{\ell+1}+\epsilon)n}(1+\|v_n\|)$ for large $n$, and
Lemma \ref{lem:genseqs} 
gives $\|\mathcal L^{(n)}_\omega v_n\|\ge 4^{-\D_{\ell}}
F_{\D_{\ell}}(\LL_\om^{(n)})\|v_n\|\ge 
e^{(\lambda_{\ell}-\epsilon)n}\|v_n\|$ for large $n$. The
conclusion follows. The proof of \eqref{it:unifliminf} is the same,
using the uniformity in Lemma \ref{lem:genseqs}.

For the inductive part, notice that the case $\ell=1$ is trivial.
Let $\ell\ge 2$ and suppose that the claims have been established for the 
case $\ell-1$. We know from \eqref{it:nrate} that elements of 
$V_\ell^{(n)}(\om)$ expand at exponential rate approximately $\lambda_\ell$
under $\LL_\om^{(n)}$. We need to show that the analogous statement holds for the 
limiting subspace $V_\ell(\om)$. We mimic the start of the proof
to control the slow $m_{\ell-1}$-codimensional subspace of $V_{\ell-1}(\om)$.
This will be exactly $V_\ell(\om)$. 

\paragraph{\emph{Proof of} \eqref{it:limsuprate}}
By the inductive hypothesis, the top Lyapunov exponent of $\LL_\om^{(n)}$ 
applied to the bundle $\{V_{\ell-1}(\om):\om \in \Om\}$ is $\lambda_{\ell-1}$ with multiplicity
$m_{\ell-1}$, with the following Lyapunov exponent being $\lambda_\ell$. 
Let $(z^{(n)}_j)_{j=1}^{m_{\ell-1}}\in 
S_X\cap V_{\ell-1}(\om)$ and $(\psi^{(n)}_j)_{j=1}^{m_{\ell-1}}$ be as guaranteed
by Lemma \ref{lem:genseqs} and let ${V_\ell'}^{(n)}(\om)=
V_{\ell-1}(\om)\cap \lin(\psi_1^{(n)},\ldots,\psi_{m_{\ell-1}}^{(n)})^\perp$. 
The same argument as in \eqref{it:nrate} shows that for arbitrary $\epsilon>0$
and sufficiently large $n$,
$\|\LL_\om^{(n)}x\|\le e^{(\lambda_\ell+\ep)n}\|x\|$ for $x\in {V_\ell'}^{(n)}(\om)$. 
The argument used in \eqref{it:Cauchy} also works, showing that 
$V_\ell'^{(n)}(\om)$ converges to a space $V'_\ell(\om)\subset V_{\ell-1}(\om)$ and,
crucially, we obtain an analogue of \eqref{eq:wbound}:
for all sufficiently large $n$, if $x\in V'_\ell(\om)$, then $x$ may be
expressed as $u+w$
with $u\in {V'_\ell}^{(n)}(\om)$ and $w\in V_{\ell-1}(\om)$ satisfying
$\|w\|\le e^{-n(\lambda_{\ell-1}-\lambda_\ell-4\epsilon)}$. 
Then, $\|\mathcal L^{(n)}_\omega u\|\lesssim e^{\lambda_\ell n}$ by 
the above;
and $\|\mathcal L^{(n)}_\omega w\|\lesssim e^{-(\lambda_{\ell-1}-\lambda_\ell)n}
\cdot e^{\lambda_{\ell-1}n}$. So $\|\mathcal L^{(n)}_\om x\|\lesssim e^{\lambda_\ell n}$
by the triangle inequality.  From \eqref{it:liminfrate},
we deduce $V'_\ell(\omega) \subseteq
V_\ell(\omega)$.  Since, by \eqref{it:inducexpts} (applied to
 $\mathcal R_{\ell-1}$), $V_\ell(\omega)$
and $V'_\ell(\omega)$ have the same finite co-dimension as subspaces of
$V_{\ell-1}(\omega)$, $V_\ell(\omega)= V'_\ell(\omega)$ and
\eqref{it:limsuprate} follows.

\paragraph{\emph{Proof of} \eqref{it:mble}}
From \eqref{it:limsuprate} and \eqref{it:liminfrate}, we see that
$V_\ell(\omega)=\{v\colon \limsup \frac 1n\log \|\mathcal
L^{(n)}_\omega v\|\le \lambda_\ell\}$ and the assumptions of Lemma \ref{lem:grassmeas} hold. Measurability of $V_\ell(\omega)$ follows.

\paragraph{\emph{Proof of} \eqref{it:inducexpts}}
Let $W(\omega)=\{w\in V_{\ell-1}(\omega)\colon 
d(w,V_\ell(\omega))\geq \frac 12\|w\|\}$.  Let $\epsilon>0$. We claim that
for sufficiently large $n$,
\begin{equation}\label{eq:minang}
  d(w',v')>e^{-\epsilon n}\text{ for all 
    $w'\in S_X\cap \mathcal L^{(n)}_\omega W(\omega)$
    and $v'\in V_\ell(\sigma^n\omega)$.}
\end{equation}

Let $\delta < \frac{\ep (\lambda_{\ell-1}-\lambda_\ell)}
{4(\lambda_{\ell-1}-\lambda_\ell+\ep)}$.  
Let $g_\ell(\om)=\sup_{p\in \N}e^{-p(\lambda_\ell+\delta)} 
\left\|\mathcal L^{(p)}_{\omega}|_{V_{\ell}(\om)}\right\|$. 
This is measurable by Lemma \ref{lem:volmeas}.
Notice that $\log^+
g_\ell(\om) \leq \log^+ \|\LL_\om \| +\max(-\lam_\ell -\del,0) +
\log^+ g_\ell(\sig\om)$.  By Lemma \ref{lem:tempered}, 
$\lim_{n\to\infty}\frac 1n \log^+ g_\ell(\sig^{n}\om)=0$.

Then, there exists $n_0(\omega)$ such that for $p,n\ge n_0$, one has
\begin{equation}\label{eq:growthbounds}
  \begin{split}
    \|\mathcal L^{(p)}_{\sigma^n\omega}z\|& \le \tfrac13 \exp(n\tfrac
    \epsilon 2 + p(\lambda_{\ell-1}+\delta))\|z\|\text{ for all $z\in
      V_{\ell-1}(\sig^n\om)$;}
    \\
    \|\mathcal L^{(p)}_{\sigma^n\omega}v\|& \le \tfrac13 \exp(n\tfrac
    \epsilon 2 + p(\lambda_{\ell}+\delta))\|v\| \text{ for all $v\in
      V_\ell(\sigma^n\omega)$}.
  \end{split}
\end{equation}
Additionally by \eqref{it:unifliminf}, $n_0$ may be
chosen so that
\begin{equation}\label{eq:growthbounds2}
  e^{n(\lambda_{\ell-1}-\delta)}<\|\mathcal L^{(n)}_\omega w\|
  <e^{n(\lambda_{\ell-1}+\delta)}\text{ for all $w\in W(\omega)\cap S_X$
   and $n\ge n_0$.}
\end{equation}

We will show \eqref{eq:minang} by contradiction.  Suppose $d(\mathcal
L^{(n)}_\omega w,V_\ell(\sigma^n\omega)) <e^{-\epsilon n}\|\mathcal
L^{(n)}_\omega w\|$ for some $n>n_0$, and such that $\epsilon
n/(\lambda_{\ell-1}-\lambda_\ell)>n_0$.  Write $\mathcal
L^{(n)}_\omega w=v+z$, with $v\in V_\ell(\sig^n\omega)$ and
$\|z\|<e^{-\epsilon n}\|v\|$. Now $\mathcal L^{(p+n)}_\omega w
=\mathcal L^{(p)}_{\sigma^n\omega}v + \mathcal
L^{(p)}_{\sigma^n\omega}z$.  Taking $p=\epsilon
n/(\lambda_{\ell-1}-\lambda_\ell)$, the bounds on the two terms coming
from \eqref{eq:growthbounds} agree, giving
$$
\|\mathcal L^{(n+p)}_\omega w\|\le e^{-\epsilon
  n/2}e^{(p+n)(\lambda_{\ell-1}+\delta)}.
$$
One checks, however, that by the choice of $\delta$, this is smaller
than $e^{(p+n)(\lambda_{\ell-1}-\delta)}$, contradicting
\eqref{eq:growthbounds2}. This establishes claim \eqref{eq:minang}.
Notice that combining \eqref{eq:minang} and \eqref{eq:growthbounds2},
we see that the restriction of the random dynamical system to the
equivariant family
$Q_{\ell-1}(\omega)=V_{\ell-1}(\omega)/V_\ell(\omega)$ satisfies for
all sufficiently large $n$, 
\begin{equation}\label{eq:quotexp}
  \|\bar{\mathcal L}^{(n)}_\omega \bar w\|_{Q_{\ell-1}(\sig^n\omega)}
  \ge e^{(\lambda_{\ell-1}-\epsilon)n}\|\bar w\|_{Q_{\ell-1}(\omega)}
  \text{ for all $\bar w\in Q_{\ell-1}(\omega)$},
\end{equation}
where $\bar{\mathcal L}_\omega$ denotes the induced action of
${\mathcal L}_\omega$ on $Q_{\ell-1}(\omega)$.

Set $\mult=\mult_{\ell-1}$. To complete the proof of
\eqref{it:inducexpts}, let $n>n_0$ be arbitrary; let
$v_1,\ldots,v_{k-\mult}$ be unit vectors in $V_\ell(\omega)$ and
$w_1,\ldots,w_{\mult}$ be unit vectors in $V_{\ell-1}(\omega)$. Then
\begin{align*}
  &d_k\mathcal L^{(n)}_\omega(v_1,\ldots,v_{k-\mult},w_1,\ldots,w_{\mult})\\
  &\ge d_{k-\mult}\mathcal L^{(n)}_\omega(v_1,\ldots,v_{k-\mult})
  d_{\mult}\bar{\mathcal L}^{(n)}_\omega(\bar w_1,\ldots,\bar
  w_{k-\mult}),
\end{align*}
where $\bar w_i$ is $w_i+V_\ell(\omega)$.  We therefore see
$$
D_k\mathcal L_\omega^{(n)}|_{V_{\ell-1}(\omega)} \ge
D_{k-\mult}\mathcal L_\omega^{(n)}|_{V_\ell(\omega)} \cdot D_\mult
\bar{\mathcal L}_\omega^{(n)}.
$$
By \eqref{eq:quotexp} and Lemma \ref{lem:unifexpDk}, $D_\mult \bar{\mathcal L}_\omega^{(n)}\gtrsim
e^{\lambda_{\ell-1}mn}$.

This gives a matching upper bound for $D_{k-\mult}\mathcal
L_\omega^{(n)}|_{V_\ell(\omega)}$ to the lower bound that we obtained
in Lemma \ref{lem:ExponentsOfReducedCocycle}.  Hence we deduce the
first $k$ exponents of $\mathcal R_{\ell-1}$ are
${\mult}={\mult}_{\ell-1}$ repetitions of $\lambda_{\ell-1}$ followed
by the first $k-{\mult}$ exponents of $\mathcal R_\ell$, establishing
\eqref{it:inducexpts}.
\end{proof}

The next corollary provides a \textsl{splitting} when the base $\sig$
is invertible and the Banach space is reflexive. Note that the methods of
\cite{GTQuas} obtain the same conclusion under the weaker assumption that 
$X^*$ has separable dual. We include this new proof, as we find it to be
illuminating.


\begin{cor}[Multiplicative Ergodic Theorem: The Oseledets splitting]
\label{cor:splitting}
Let $\mathcal R$ be a random dynamical system acting on a reflexive separable
Banach space. Suppose that the base, $\sig$, is invertible; and that
$\kappa(\mathcal R)<\lambda(\mathcal R)$.
Then there exist $1\le r\le\infty$ and exceptional Lyapunov exponents
and multiplicities as in Theorem \ref{thm:Osel}. Furthermore, there is
a measurable direct sum decomposition
\footnote{In the case $r=\infty$, the decomposition is
$X=\bigoplus_{i=1}^\infty Z_i(\om)\oplus V_\infty(\om)$}
\begin{equation*}
X=Z_1(\omega)\oplus \dots \oplus Z_r(\omega)\oplus V_\infty(\omega),
\end{equation*}
such that for $\PP$-a.e.\ $\omega$, $\LL_\om(Z_i(\om))=Z_i(\sig(\om))$ for
each $i$, $\LL_\om(V_\infty(\om))\subset V_\infty(\sig(\om))$, 
 $\dim Z_i(\omega)=m_i$ and
$\lim_{n\to\infty}
\frac1n\log\|\LL^{(n)}_\omega v\|=\lambda_i$ for 
$v\in Z_i(\omega)\setminus\{0\}$; $\limsup_{n\to\infty} \frac1n\log\|\LL^{(n)}_\omega v\|\le \kappa(\mathcal R)$ for $v\in V_\infty(\omega)$.
\end{cor}

We make use of
the following facts valid for reflexive Banach spaces.  If $X$ is
reflexive and $\Theta$ is a closed subspace of $X^*$ of codimension
$k$, then its annihilator, $\Theta^\perp$ is $k$-dimensional. Further
if $\theta$ is a bounded functional such that
$\theta|_{\Theta^\perp}=0$, then $\theta\in\Theta$.

\begin{proof}
  Let $\mc{R}^*$ be the dual random dynamical system to $\mc{R}$ as
  defined above.  Applying Theorem~\ref{thm:Osel} to $\mc{R}^*$, and
  recalling from Theorem~\ref{thm:summ} that the Lyapunov exponents
  and multiplicities of $\mc{R}$ and $\mc{R}^*$ coincide, yields a
  $\sig^{-1}$ equivariant measurable filtration $X^*=V^*_1(\om)\supset
  \dots \supset V^*_r(\omega)\supset V^*_\infty(\om)$, with the same
  codimensions as those of $\mc{R}$.

  Let $Y_\ell(\omega)={V_\ell^*(\omega)}^\perp$. Notice that $\dim
  Y_\ell(\om)=\D_{\ell-1}$.  Since $V^*_{\ell}(\om)$ is measurable and
  $(\cdot) ^ \perp : \mc{G}(X^*) \to \mc{G}(X)$ is continuous \cite[IV
  \S2]{Kato}, then $Y_\ell(\om)$ is measurable.  Also, for every $\psi
  \in V^*_{\ell}(\sig\om)$, we have $\mathcal L_{\sig\om}^*\psi \in
  V^*_{\ell}(\om)$ by equivariance of $V^*_{\ell}(\cdot)$.  Hence, for
  every $y \in Y_\ell(\om)$, $0=\mathcal L_{\sig\om}^* \psi(y)=\psi
  (\mathcal L_{\om} y)$. Thus, $\mathcal L_{\om} Y_\ell(\om) \subset
  Y_\ell(\sig\om)$, yielding equivariance.

  
We define $Z_\ell(\omega)=Y_{\ell+1}(\omega)\cap V_\ell(\omega)$. It remains to show that
$V_{\ell-1}(\omega)=V_\ell(\omega)\oplus Z_{\ell-1}(\omega)$. To prove this, it suffices
to show that $V_\ell(\omega)\oplus Y_\ell(\omega)=X$. Suppose this is not the case. Then,
there exists $v\in V_\ell(\omega)\cap Y_\ell(\omega)\cap S_X$. Let $\theta\in S_{X^*}$ be
such that $\theta(v)=1$. Let $\bar\theta$ be the equivalence class of $\theta$
in $Q^*(\omega)=X^*/V_\ell^*(\omega)$. 

We record a corollary of (12). 
For almost every $\omega$, one has for all large $n$
$$
\|\bar{\mathcal L}^{*(n)}_{\omega}\bar\psi\|_{Q^*(\sigma^{-n}\omega)}\ge 
e^{-(\lambda_{\ell-1}-\epsilon)n}\|\bar\psi\|_{Q^*(\omega)}
\text{ for all $\bar\psi\in Q^*(\omega)$}.
$$
Since $Q^*(\omega)$ is a finite-dimensional space whose dimension does 
not depend on $\omega$, the above implies that 
$\bar{\mathcal L}^*_\omega$ is bijective. Furthermore, the quantity
$$
C(\omega)=\inf_{n\in\mathbb N; \bar\psi\in S_{X^*}\cap Q^*(\omega)}
e^{-(\lambda_{\ell-1}-\epsilon)n}\|\bar{\mathcal L}_\omega^{*(n)}\bar\psi\|
_{Q^*(\sigma^{-n}\omega)}
$$
is positive. We claim that $C(\omega)$ is measurable. 
Let $(\zeta_n(\omega))$ be a measurable
dense subsequence of $V_\ell^*(\omega)$. If $\bar\psi(\omega)$ is the 
equivalence class of $\psi$ in $Q^*(\omega)$, then we have
$\|\bar\psi(\omega)\|_{Q^*(\omega)}=\inf_k \|\psi(\omega)-\zeta_k(\omega)\|$, which
depends measurably on $\omega$. Proceeding as in Lemma \ref{lem:volmeas},
we see that $C(\omega)$ is measurable and by \eqref{eq:quotexp}
is positive almost everywhere. Hence $C(\omega)$ exceeds some quantity $c$ on a set of
positive measure. 

Let $\bar\phi_n\in Q^*(\sigma^n\omega)$ be such that $\bar{\mathcal L}_{\sigma^n\omega}
^{*(n)}\bar\phi_n=\bar\theta$. Then
$$
\|\bar{\mathcal L}^{*(n)}_{\sigma^n\omega}\bar\phi_n\|_{Q^*(\omega)}\ge
C(\sigma^n\omega)e^{(\lambda_{\ell-1}-\epsilon)n}\|\bar\phi_n\|_{Q^*(\sigma^n\omega)}.
$$
By ergodicity, there exist arbitrarily large values of $n$ for which
\begin{equation}\label{eq:last}
\|\bar\phi_n\|_{Q^*(\sigma^n\omega)}
\le c^{-1}e^{-(\lambda_{\ell-1}-\epsilon)n}\|\bar\theta\|_{Q^*(\omega)}. 
\end{equation}
On the other hand, 
one has $v\in Y_l(\omega)$, so that $\psi(v)=0$ for every
$\psi\in V_l^*(\omega)$. 
Thus, if we express $\mathcal L_{\sigma^n\omega}^{*(n)} \phi_n + \psi_n =\theta$, where
$\phi_n\in X^*$ is a representative of $\bar\phi_n$ and
$\psi_n \in V_l^*(\omega)$, the following holds
$$
1=\theta(v)= \mathcal L_{\sigma^n\omega}^{*(n)} \phi_n(v) + \psi_n(v)=
\phi_n(\mathcal L^{(n)}_\omega v).$$
In addition, for every $\psi \in V_l^*(\sigma^{-n}\omega)$, 
$(\phi_n+\psi)(\mathcal L^{(n)}_\omega v)= \phi_n(\mathcal L^{(n)}_\omega v)=1$.
Thus, for sufficiently large $n$ and every $\psi \in V_l^*(\sigma^{-n}\omega)$,  $\|\phi_n+\psi\| e^{(\lambda_\ell+\epsilon)n}\ge 1.$
Therefore, $\|\bar\phi_n\|_{Q^*(\sigma^n(\omega))}\ge e^{-(\lambda_\ell+\epsilon)n}$, giving 
a contradiction with \eqref{eq:last}. Hence, $V_{\ell-1}(\omega)=V_\ell(\omega)\oplus Z_{\ell-1}(\omega)$ as
required. 

\end{proof}

\subsubsection*{Acknowledgments}
CGT acknowledges support from Australian Research Council Discovery
Project DP110100068 at UNSW. AQ acknowledges support from the Canadian NSERC,
and thanks the Universidade de S\~ao Paulo for the invitation to
deliver a mini-course from which this work originated.

The authors would like to thank the referees for a very careful reading and
helpful suggestions.
\bibliographystyle{abbrv}

\end{document}